\documentclass{amsart}
\usepackage{graphicx}
\usepackage{newlattice}

\DeclareMathOperator{\Princ}{Princ}

\newcommand{\fp}{\F p}
\newcommand{\fq}{\F q}

\newcommand{\Cll}{C_\tup{ll}}
\newcommand{\Cul}{C_\tup{ul}}
\newcommand{\Clr}{C_\tup{lr}}
\newcommand{\Cur}{C_\tup{ur}}
 
\newcommand{\lcorner}[1]{\textup{lc(#1)}}
\newcommand{\rcorner}[1]{\textup{rc(#1)}}

\theoremstyle{plain}
\newtheorem{theorem}{Theorem}

\newtheorem{lemma}[theorem]{Lemma}

\theoremstyle{definition}

\begin{document}
\title[An extension theorem for planar semimodular lattices]
{An extension theorem\\for planar semimodular lattices}  
\author{G. Gr\"{a}tzer} 
\address{Department of Mathematics\\
  University of Manitoba\\
  Winnipeg, MB R3T 2N2\\
  Canada}
\email[G. Gr\"atzer]{gratzer@me.com}
\urladdr[G. Gr\"atzer]{http://server.maths.umanitoba.ca/homepages/gratzer/}

\author{E.\,T. Schmidt}
\address{Mathematical Institute of the Budapest University of
        Technology and Economics\\  
        H-1521 Budapest\\
        Hungary}
\email[E.\,T. Schmidt]{schmidt@math.bme.hu}
\urladdr[E.\,T. Schmidt]{http://www.math.bme.hu/\~{}schmidt/}

\date{Aug. 26, 2013}

\begin{abstract}
We prove that every finite distributive lattice $D$
can be represented as the congruence lattice of 
a rectangular lattice $K$ in which all congruences are principal.
We verify this result in a stronger form as an extension theorem.
\end{abstract}

\dedicatory{To L\'aszlo Fuchs,\\
our teacher,\\
on his 90th birthday}
\maketitle

\section{Introduction}\label{S:Introduction}
In G. Gr\"atzer and E.\,T. Schmidt~\cite{GS62}, 
we proved that every finite distributive lattice $D$
can be represented as the congruence lattice of 
a sectionally complemented finite lattice $K$.
In such a lattice, of course, all congruences are principal, 
using the notation of G. Gr\"atzer~\cite{gG13a},
$\Con K = \Princ K$.

Since every finite distributive lattice $D$
can be represented as the congruence lattice of 
a planar semimodular lattice $K$ 
(see G. Gr\"atzer, H. Lakser, and E.\,T. Schmidt~\cite{GLS98a}),
it is reasonable to ask whether instead of the 
sectional complemented lattice of the previous paragraph,
we can construct a planar semimodular lattice~$K$.

G. Gr\"atzer and E. Knapp~\cite{GKn09} proved a result 
stronger than the Gr\"atzer--Lakser--Schmidt result:
every finite distributive lattice $D$
can be represented as the congruence lattice of 
a rectangular lattice $K$---see Section~\ref{S:Rectangular} 
for the definition. 
(For a new proof of this result, 
see G. Gr\"atzer and E.\,T. Schmidt~\cite{GS13b}.)
Keeping this in mind, we prove:

\begin{theorem}\label{T:mainrepr}
Every finite distributive lattice $D$
can be represented as the congruence lattice of 
a rectangular lattice $K$ with the property that
all congruences are principal.
\end{theorem}

We prove this  representation result in a much stronger form, 
as an extension theorem.

\begin{theorem}\label{T:mainextension}
Let $L$ be a planar semimodular lattice. 
Then $L$ has an extension $K$ satisfying the following conditions:
\begin{enumeratei}

\item $K$ is a rectangular lattice;

\item $K$ is a congruence-preserving extension of $L$;

\item $K$ is a cover-preserving extension of $L$;

\item every congruence relation of $K$ is principal.

\end{enumeratei}
\end{theorem}

Observe that we only have to prove 
Theorem~\ref{T:mainextension}.
Indeed, let Theorem~\ref{T:mainextension} hold and 
let $D$ be a finite distributive lattice. 
By G. Gr\"atzer and E. Knapp~\cite{GKn09}, 
there is a planar semimodular lattice $K_1$ whose congruence lattice is isomorphic to $D$. 
By~Theorem~\ref{T:mainextension},
the lattice $K_1$ has a congruence-preserving extension $K$ 
in which every congruence relation is principal. 
This lattice $K$ satisfies the conditions of Theorem~\ref{T:mainrepr}.

We will use the notations and concepts of lattice theory 
as in \cite{LTF}.
See \cite{CFL} for a deeper coverage of finite congruence lattices. 
See G. Cz\'edli and G.~Gr\"atzer~\cite{CGa} and
G. Gr\"atzer~\cite{gG14a} for an overview of semimodular lattices,
structure and congruences.

\section{Background}\label{S:Background}
We need some concepts and results from the literature 
to prove Theorem~\ref{T:mainextension}.

\subsection{Rectangular lattices}\label{S:Rectangular}

Let $L$ be a planar lattice. 
A \emph{left corner} (resp., \emph{right corner}) 
of the lattice $L$ is a doubly-irreducible element in $L - \set{0,1}$ 
on the left (resp., right) boundary of~$L$. 
A \emph{corner} of $L$ is an element in $L$ 
that is either a left or a right corner of $L$. 
G. Gr\"atzer and E. Knapp~\cite{GKn09} 
define a \emph{rectangular lattice} $L$ 
as a planar semimodular lattice 
which has exactly one left corner, $\lcorner L$, 
and exactly one right corner, $\rcorner L$, 
and they are complementary---that is, 
$\lcorner L \jj \rcorner L = 1$ 
and $\lcorner L \mm \rcorner L = 0$. 
In a rectangular lattice $L$, 
there are four boundary chains: 
the lower left, the lower right, 
the upper left, and the upper right, 
denoted by $\Cll(L)$, $\Clr(L)$, $\Cul(L)$, 
and $\Cur(L)$, respectively. 

Let $A$ and $B$ be rectangular lattices. 
We define the \emph{rectangular gluing} of $A$ and $B$ 
as the gluing of $A$ and $B$ over the ideal $I$ and filter $J$, 
where $I$ is the lower left boundary chain of $A$ and
$J$ is the upper right boundary chain of $B$ (or symmetrically). 

We recap some basic facts about rectangular lattices 
(G. Gr\"atzer and E. Knapp~\cite{GKn09} and~\cite{GKn10},
G. Cz\'edli and E.\,T. Schmidt~\cite{CSa} and~\cite{CSb}). 

\begin{theorem}\label{T:rectang}
Let $L$ be a rectangular lattice.
\begin{enumeratei}
\item The ideal $\id {\lcorner L}$ is the chain $\Cll(L)$, and symmetrically.
\item The filter $\fil {\lcorner L}$ is the chain $\Cul(L)$, and symmetrically.
\item For every $a \leq \lcorner L$, 
the interval $[a, \rcorner L \jj a]$ is a chain, and symmetrically. 
\item For every $a \leq \lcorner L$, 
$L$ is a rectangular gluing of the filter $\fil a$
and the ideal $\id{\rcorner L \jj a}$.
\item For every prime interval $\fp$ of the chain 
$[a, \rcorner L \jj a]$, there is a prime interval~$\fq$
of the chain $\Clr$ so that $\fp$ and $\fq$ are perspective.
\end{enumeratei}
\end{theorem}

Note that it follows from (v) that
\[
   \con{\Cul} =  \con{a, \rcorner L \jj a} 
   = \con{\Clr}.
\]

\subsection{Eyes}\label{S:Eyes}
Let $L$ be a planar lattice.
An interior element of an interval of length two 
is called an \emph{eye} of $L$.
We will \emph{insert} and \emph{remove} eyes in the obvious sense.
A planar semimodular lattice $L$ is \emph{slim} 
if it has no eyes.

\subsection{Forks}\label{S:Forks}
We need from G. Cz\'edli and E.\,T. Schmidt~\cite{CSb}
the fork construction.

\begin{figure}[b]
\centerline{\includegraphics{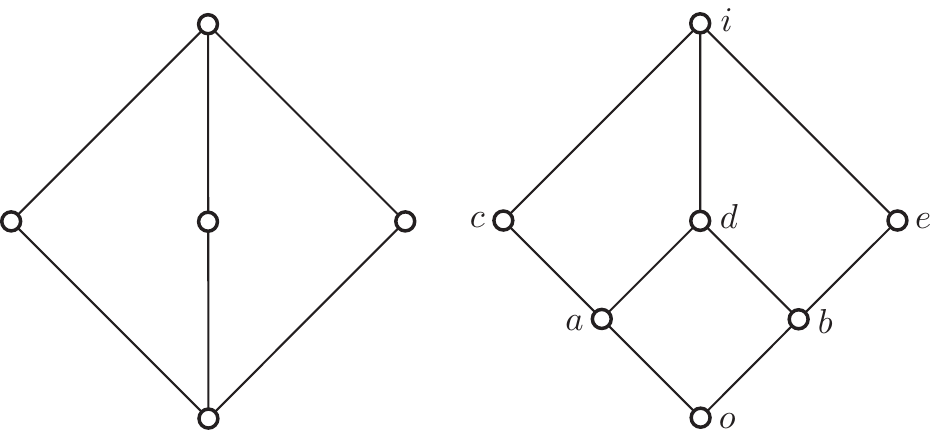}}
\caption{The lattices $\SM 3$ and $\SfS 7$}\label{F:s7}
\end{figure}

\begin{figure}[b]
\centerline{\includegraphics{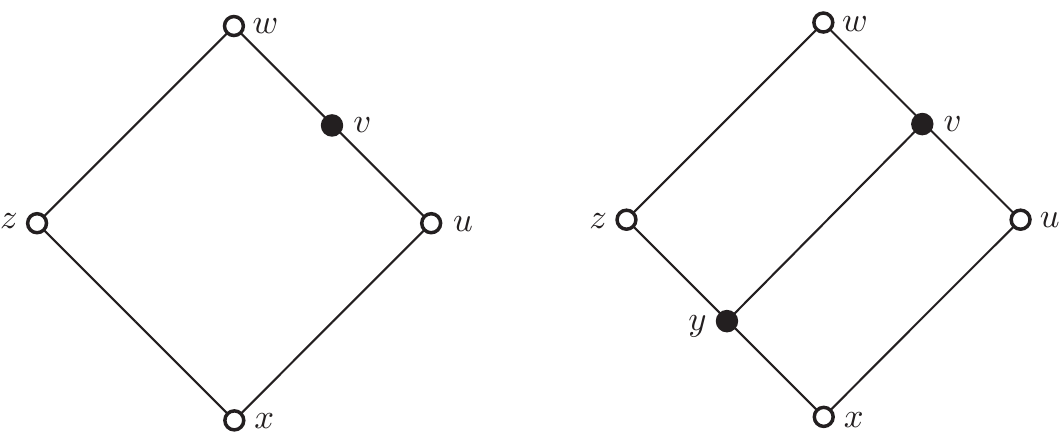}}
\caption{A step in inserting a fork}\label{F:step}
\end{figure}

\begin{figure}[t]
\centerline{\includegraphics{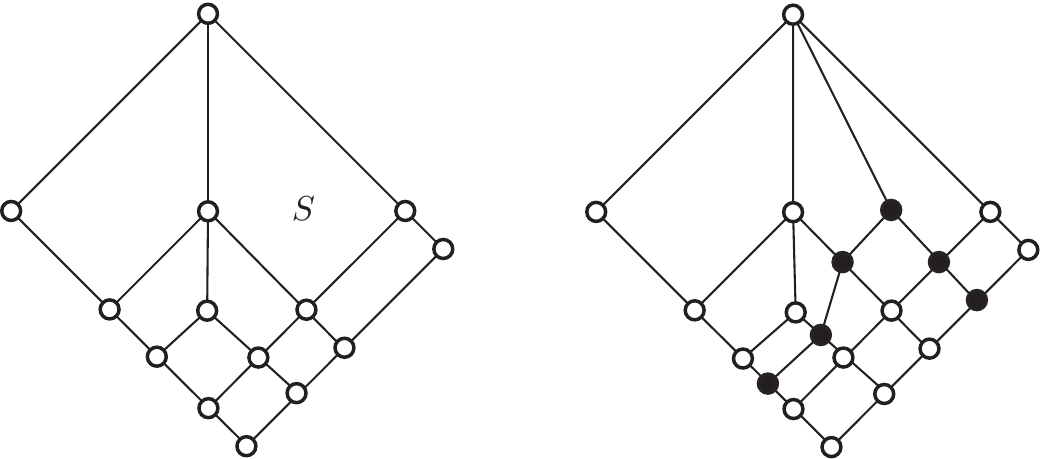}}
\caption{Inserting a fork at $S$}\label{F:forks}
\end{figure}

Let $L$ be a planar semimodular lattice. Let $L$ be slim. 
\emph{Inserting a fork} into $L$ at the covering square $S$, firstly, 
replaces $S$ by a~copy of $\SfS 7$. 
We get three new covering squares replacing~$S$ of $L$.
We will name the elements of the inserted $\SfS 7$ as in Figure~\ref{F:s7}.

Secondly, if there is a chain 
$u\prec v\prec w$ such that~the element $v$ has just been inserted 
(the element $a$ or $b$ in $\SfS 7$ in the first step) and 
$
   T = \set{x=u \mm z, z, u, w=z \jj u}
$
is a covering square in the lattice~$L$ 
(and so $u \prec v \prec w$ is not on the boundary of $L$) 
but $x \prec z$ at the present stage of the construction,
then we insert a new element~$y$ into the interval $[x, z]$
such that $x \prec y \prec z$ and $y \prec v$, 
see Figure~\ref{F:step}. 
We get two covering squares to replace the covering square~$T$.

Let $K$ denote the lattice we obtain when the procedure terminates 
(that is, when the new element is on the boundary);
see Figure~\ref{F:forks} for an example. 

The new elements form an order, called a \emph{fork} 
(the black filled elements in Figure~\ref{F:forks}).
We say that $K$ is obtained from
$L$ by \emph{inserting a fork into $L$} at the covering square~$S$.

Here are some basic facts, based on G.~Cz\'edli and E.\,T.~Schmidt~\cite{CSb}, about this construction.
 
\begin{lemma}\label{L:forks}
Let $L$ be a planar semimodular lattice 
and let $S$ be a covering square in~$L$.
If $L$ is slim, then inserting a fork into $L$ at~$S$ 
we obtain a slim planar semimodular lattice $K$.
If $L$ is rectangular, so is $K$.

If $y$ is an element of the fork outside of $S$, 
then $[y_*, y]$ is up-perspective to $[o, a]$ or $[o, b]$, 
where $y_*$ is the lower cover of $y$ in $K - L$.
\end{lemma}

\subsection{Patch lattices}\label{S:Patch}
Let us call a rectangular lattice $L$ a \emph{patch lattice} 
if $\lcorner A$ and $\rcorner A$ are dual atoms; 
Figure~\ref{F:s7} has two examples.
The next lemma is a trivial application of Lemma~\ref{L:forks}.

\begin{lemma}\label{L:patch}
Let $L$ be a slim patch lattice 
and let $S$ be a covering square in $L$.
Inserting a fork into $L$ at~$S$, we obtain a slim patch lattice $K$.
\end{lemma}

\subsection{The structure theorems}\label{S:structure}
Now we state the structure theorems for patch lattices and rectangular lattices
of G.~Cz\'edli and E.\,T.~Schmidt~\cite{CSb}.

\begin{theorem}\label{T:patchstructure} 
Let $L$ be a patch lattice. 
Then we can obtain $L$ from the four-element Boolean lattice $\SC 2^2$
by first inserting forks, then inserting eyes.
\end{theorem}

\begin{theorem}\label{T:rectangularstructure} 
Let $L$ be a rectangular lattice. Then there is a sequence of lattices
\[
   K_1, K_2, \dots, K_n = L
\]
such that each $K_i$, for $i = 1, 2, \dots, n$, 
is either a patch lattice or it is the rectangular gluing of the lattices $K_j$ and $K_k$ for $j, k < i$. 
\end{theorem}

See also G. Gr\"atzer and E. Knapp~\cite{GKn10} and G.~Gr\"atzer~\cite{gG13}.

\subsection{A congruence-preserving extension}
\label{S:congruence-preserving}
Finally, we need the following result of 
G. Gr\"atzer and E. Knapp \cite{GKn09}.

\begin{theorem}\label{T:extension}
Let $L$ be a planar semimodular lattice. 
Then there exists a rectangular, cover-preserving, and 
congruence-preserving extension $K$ of $L$.
\end{theorem}

\section{congruences of rectangular lattices}\label{S:cong}

To prove Theorem~\ref{T:mainextension}, 
we need a ``coodinatization'' 
of the congruences of rectangular lattices.

\begin{theorem}\label{L:patchcongs}
Let $L$ be a rectangular lattice and let $\bga$ be a congruence of $L$.
Let~$\bga^l$ denote the restriction of $\bga$ to $\Cll$. 
Let $\bga^r$ denote the restriction of $\bga$ 
to $\Clr$. 
Then the congruence $\bga$ is determined 
by the pair $(\bga^l,\bga^r)$. In fact, 
\[
   \bga = \con{\bga^l \uu \bga^r}.
\]
\end{theorem}

\begin{proof}
Since $\bga \geq \con{\bga^l \uu \bga^r}$, 
it is sufficient to prove that
\begin{enumeratei}
\item[(P)] if the prime interval $\fp$ of $A$ is collapsed 
by the congruence $\bga$, 
then it is collapsed 
by the congruence $\con{\bga^l \uu \bga^r}$.

\end{enumeratei}

First, let $L$ be a slim patch lattice.
By Theorem~\ref{T:patchstructure}, we obtain $L$ from the square, $\SC 2 ^2$, with a sequence of $n$ fork insertions. 
We induct on $n$.

If $n = 0$, then $L = \SC 2 ^2$, and the statement is trivial.

Let the statement hold for $n - 1$ and let $K$ be the patch lattice
we obtain by $n - 1$ fork insertions into $\SC 2 ^2$,
so that we obtain $L$ from $K$
by one fork insertion at the covering square $S$.
We have three cases to consider.

Case 1. $\fp$ is a prime interval of $K$. 
Then the statement holds for $\fp$ and $\bga\restr_K$, 
the restriction of $\bga$ to $K$ by induction.
So $\fp$ is collapsed by $\con{(\bga\restr_K)^l \uu (\bga\restr_K)^r}$
in $K$. Therefore, (P) holds in $L$.

In the next two cases, 
we assume that $\fp$ is not in $K$.

Case 2. $\fp$ is perspective to a prime interval of $K$. 
Same proof as in Case 1. This case includes $\fp = [o,a]$ and 
any one of the new intervals up-perspective with $[o,a]$.

Case 3.  $\fp = [a,c]$ and  
any one of the new intervals is up-perspective with $[a,c]$.
Then the fork extension defines 
the terminating prime interval $\fq = [y,z]$ on the boundary
of $L$ which is up-perspective with $\fp$, verifying (P).

Secondly, let $L$ be a patch lattice, not necessarily slim.
This case is obvious because (P) is preserved when inserting an eye.

Finally, if $L$ is not a patch lattice, we induct on $|L|$.
By Theorem~\ref{T:rectangularstructure}, $L$ is the rectangular gluing 
of the rectangular lattices $A$ and $B$ 
over the ideal $I$ and filter~$J$. 
Let $\fp$ be a prime interval of $L$. 
Then $\fp$ is a prime interval of $A$ or~$B$, say, of $A$.
(If $\fp$ is a prime interval of $B$, then the argument is easer.)
By~induction, $\fp$ is collapsed by 
$\con{\,\bga\restr_{\Cll(A)\,} \uu \bga\restr_{\Clr(A)}}$,
so it is collapsed by $\con{\,\bga\restr_{\Cll(L)} \uu \bga\restr_{\Clr(L)}\,} = \con{\bga^l \cup \bga^r}$.
\end{proof}

\section{Construction}\label{S:Construction}

Now we proceed with the construction 
for the planar semimodular lattice $L$ for Theorem~\ref{T:mainextension}.

\tbf{Step 1.}
We apply Theorem~\ref{T:extension} to get  
a rectangular, cover-preserving, and 
congruence-preserving extension $K_1$ of $K$.

\tbf{Step 2.}
Let $D = \Clr(K_1)$. 
We form the lattice $D^2$, 
and insert eyes into the covering squares of the main diagonal,
obtaining the lattice $\widehat D$, see Figure~\ref{F:cons2}.

Now we do a rectangular gluing of $K_1$ and $\widehat D$, 
obtaining the lattice $K_2$.
Here is the crucial statement:

\begin{lemma}\label{L:K2}
$K_2$ is a rectangular, cover-preserving, and 
congruence-preserving extension of $L$.
For every join-irreducible congruence $\bga$ of $L$, 
there is a prime interval $\fp_{\bga}$ of $C = \Cll(K_2)$
such that $\con{\fp_{\bga}}$ in $K_2$ 
is the unique extension of $\bga$ to $K_2$.
\end{lemma}

\begin{proof}
Indeed, by Theorem~\ref{L:patchcongs}, 
there is a prime interval $\fq_{\bga}^l$ of $\Cll(K_1)$
or a prime interval $\fq_{\bga}^r$ of $\Clr(K_1)$
such that $\con{\fq_{\bga}^l}$ or $\con{\fq_{\bga}^r}$ in $K_1$ 
is the unique extension of $\bga$ to~$K_1$.
If we have  $\fq_{\bga}^l \ci \Cll(K_1) \ci C$, 
set $\fq_{\bga}^l= \fp_{\bga}$ and we are done.

If we have $\fq_{\bga}^r \ci \Clr(K_1)$ 
with $\con{\fq_{\bga}^r}$ the unique extension of $\bga$ to~$K_1$,
then in~$K_2$ there is a unique $\fq \ci \Cll(\widehat D) \ci \Cll(K_2)$
such that in $\widehat D$, the prime intervals $\fq_{\bga}^r$ and $\fq$ 
are connected by an $\SM 3$ on the main diagonal; 
see Figure~\ref{F:cons2expl} for an illustration.

Now clearly, we can set $\fp_{\bga} = \fq$.
\end{proof}
 
Note: Lemma \ref{L:K2} is a variant of several published results.
Maybe G. Cz\'edli \cite[Lemma 7.2]{gC12} is its closest predecessor.

\begin{figure}[t]
\centerline{\includegraphics{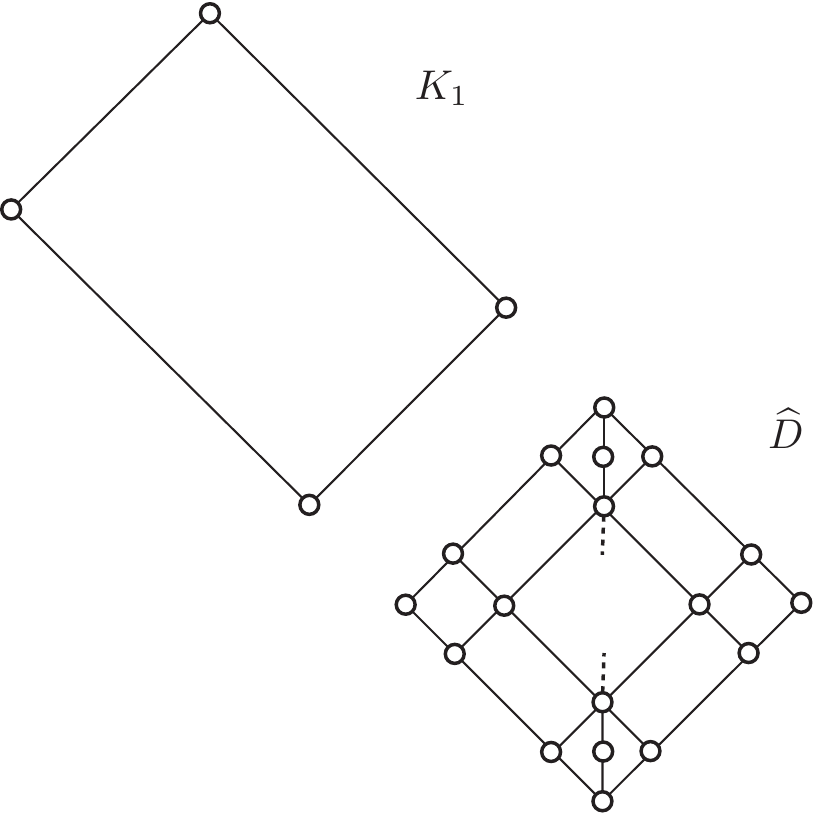}}
\caption{Step 2 of construction}\label{F:cons2}
\end{figure}

\begin{figure}[t]
\centerline{\includegraphics{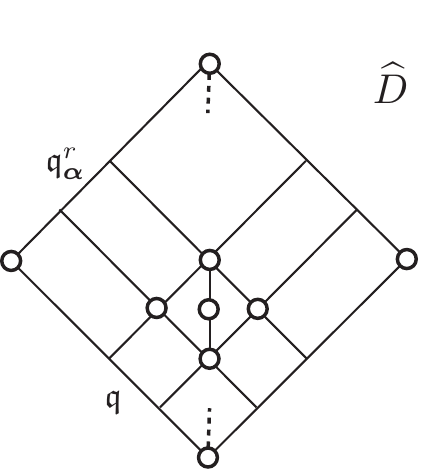}}
\caption{Step 2 of construction: a detail of the lattice $\widehat D$}\label{F:cons2expl}
\end{figure}

\tbf{Step 3.}
For the final step of the construction,
take the chain  $C = \Cll(K_2)$ and a~congruence $\bga$ of $L$.
We can view $\bga$ as a congruence of $K_2$ and let
$
   \bga = \bgg_1 \jj \cdots \jj \bgg_n
$ 
be a join-decomposition of $\bga$ into join-irreducible congruences. 
By Theorem~\ref{L:patchcongs} and (P), 
we can associate with each $\bgg_i$, for $i = 1, \dots, n$, 
a prime interval $\fp_i$ of $C$ so that $\con{\fp_i} = \bgg_i$. 

We construct a rectangular lattice $C[\bga]$ 
(a cousin of $\widehat D$) as follows:

Let $\SC {n+1} = \set{0 < 1 < \dots < n}$. 
Take the direct product $C \times \SC {n+1}$.
We think of this direct product as consisting of $n$ columns,
column $1$ (the bottom one), \dots, column $n$ (the top one). 

In column $i$, for $1 \leq i \leq n$, 
we take the covering square 
whose upper right edge is perspective to~$\fp_i$ and insert an eye. 
In the covering $\SM 3$ sublattice we obtain, 
every prime interval~$\fp$ satisfies $\con \fp = \bgg_i$. 
See Figure~\ref{F:cons3} for an illustration with $n = 3$; 
a prime interval~$\fp$ is labelled with $\bgg_i$ 
if $\con \fp = \bgg_i$.

Let $b$ denote the top element of the $\SM 3$ 
we constructed for $\fp_n$, clearly, we have $b \in \Cur(C[\bga])$.
Take the element $a \in \Cll(C[\bga])$ 
so that the interval $[a, b]$ is a chain of length $n$. 
Then the $n$ prime intervals $\fq_1$, \dots, $\fq_{n}$ 
of $[a, b]$ satisfy  
\[\con {\fq_1} = \bgg_1, \dots, \con{\fq_{n}} = \bgg_n,\]  
so $\con{[a, b]} = \bga$,
finding that in the lattice $C[\bga]$, 
the congruence $\bga$ is principal.

We identify $C$ with $\Cur(C[\bga])$; 
note that this is a ``congruence preserving'' isomorphism: 
for a prime interval  $\fp$ of $C$, 
the image $\fp'$ of $\fp$ in $\Cur(C[\bga])$ 
satisfies $\con{\fp} = \con{\fp'}$.

\begin{figure}[b]
\centerline{\includegraphics{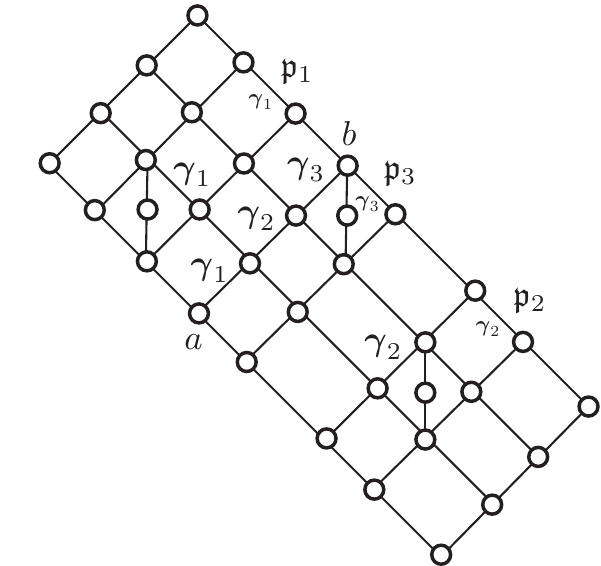}}
\caption{Step 3 of construction: the lattice $C[\bga]$}\label{F:cons3}
\end{figure}

Now we form the rectangular gluing of $C[\bga]$ with filter $C$ 
and $K_2$ with the ideal~$C$ to obtain the lattice $K_2[\bga]$.
Obviously, $K_2[\bga]$ is a rectangular lattice,
it is a cover-preserving 
congruence-preserving extension of $K_2$ and, therefore, of $L$.

It is easy to see that $\Cll(K_2[\bga])$ is still 
(congruence) isomorphic to $C$; for a rigorous treatment
see the Corner Lemma and the Eye Lemma in G. Cz\'edli~\cite{gC12}
as they are used in the proof of \cite[Lemma 7.2]{gC12}. 
We can continue this expansion 
with all the congruences of $L$. 
In the last step, we get the lattice $K_3 = K$, 
satisfying all the conditions of Theorem~\ref{T:mainextension}.

\subsection{Discussion}\label{S:Discussion}
Let $L$ be a rectangular lattice and 
let $\bga$ be a join-irreducible congruence of $L$.
We call $\bga$ \emph{left-sided}, 
if there a prime interval $\fp \ci \Cll(L)$ 
such that $\con{\fp} = \bga$ 
but there is no such $\fp \ci \Clr(L)$. 
In the symmetric case, 
we call $\bga$ \emph{right-sided}.
The congruence $\bga$ is \emph{one-sided} 
if it is left-sided or right-sided. 
The~congruence $\bga$ is \emph{two-sided} 
if it is not one-sided.

Using these concepts, we can further analyze 
Theorem~\ref{L:patchcongs} and condition (P).
By~Theorems \ref{T:patchstructure} and \ref{T:rectangularstructure}, 
we build a rectangular lattice from a grid 
(the direct product of two chains) by inserting first forks and then eyes.
At~the start, all join-irreducible congruences
are one-sided. 
When we insert a fork, we introduce a two-sided congruence.
When we insert an eye, we identify two congruences, 
resulting in a two-sided congruence.

What congruence pairs occur in Theorem~\ref{L:patchcongs}? 
Let $\bgb_l$ be a congruence of $\Cll(L)$ and 
let $\bgb_r$ be a congruence of $\Clr(L)$.
Under what conditions is there a congruence~$\bga$ of $L$
such that $\bga^l = \bgb_l$ and $\bga^r = \bgb_r$?
Here is the condition: 
If $\fp$ is a prime interval of $\Cll(L)$ collapsed by $\bgb_l$  and there is a prime interval $\fq$ of $\Clr(L)$ with $\con \fp = \con \fq$, 
then $\fq$ is collapsed by $\bgb_r$; and symmetrically.

In Step 3 of the construction, we use the chain $\SC {n+1}$.
Clearly, $\SC n$ would have sufficed. Can we use, in general, shorter chains?

In a finite sectionally complemented lattice, 
the congruences are determined around the zero element. 
So it is clear that 
for finite sectionally complemented lattices, 
all congruences are principal.

For a finite semimodular lattice, 
the congruences are scattered all over. 
So it is somewhat surprising 
that Theorem~\ref{T:mainrepr} holds.

For modular lattices, the situation is similar 
to the semimodular case. 
E.\,T. Schmidt \cite{tS62} proved that every 
finite distributive lattice $D$
can be represented as the congruence lattice of 
a countable modular lattice $K$. 
(See also G. Gr\"atzer and E.\,T. Schmidt \cite{GS99} 
and \cite{GS03a}.) It is an interesting question whether
Theorem~\ref{T:mainrepr} holds for
countable modular lattices.

The congruence structure of planar semimodular lattices 
is further investigated in three recent papers: 
G.~Cz\'edli~\cite{gCa}, \cite{gCb} and G.~Gr\"atzer~\cite{gG13c}.

\end{document}